\documentclass[12pt,reqno]{amsart}

\usepackage{amssymb}
\usepackage{amsmath}
\usepackage{amsthm}
\usepackage{amsfonts}
\usepackage{amscd}
\usepackage{graphicx}

\usepackage{ifthen}

\usepackage{fullpage}
\usepackage{float}

\usepackage{graphics}
\usepackage{latexsym}
\usepackage{epsf}
\usepackage{breakurl}
\usepackage{indentfirst}
\usepackage{bigints}

\usepackage[utf8]{inputenc}
\usepackage[T1]{fontenc}
\usepackage{mathrsfs}
\usepackage{epsfig}
\usepackage{mathtools}
\usepackage{multicol}
\usepackage{enumitem}

\allowdisplaybreaks

\newcommand{\nnend}{\nonumber\\}

\title{On the rate of escape or approach to the origin of a random string}
\author{Phuc Lam}
\address{
Department of Mathematics \\
University of Rochester \\
Rochester, NY, 14627 \\
USA \\}
\email{plam6@u.rochester.edu}
\thanks{Supported by a Discover Grant for Undergraduate Summer Research at the University of Rochester}
\date{}

\begin{document}
\theoremstyle{plain}
\newtheorem{theorem}{Theorem}[section]
\newtheorem{conjecture}[theorem]{Conjecture}
\newtheorem{corollary}[theorem]{Corollary}
\newtheorem{lemma}[theorem]{Lemma}
\newtheorem{proposition}[theorem]{Proposition}

\theoremstyle{definition}
\newtheorem{example}[theorem]{Example}
\newtheorem{definition}[theorem]{Definition}
\newtheorem{notation}[theorem]{Notation}

\theoremstyle{remark}
\newtheorem{remark}[theorem]{Remark}

\numberwithin{equation}{section}
\maketitle

\begin{abstract}
    In this paper, we extend upon a result by Mueller and Tribe regarding Funaki's model of a random string. Specifically, we examine the rate of escape of this model in dimensions $d \ge 7$. We also provide a bound for the rate of approach to the origin in dimension $d = 6$.
\end{abstract}


\section{Introduction and main results}\label{sec:intro}
Mueller and Tribe studied recurrence questions for the following model of a random string in \cite{MT02}:
\begin{equation}\label{eq:FunakiModel}
    \dfrac{\partial u_t(x)}{\partial t} = \dfrac{\partial^2 u_t(x)}{\partial x^2} + \dot{W}(x, t),
\end{equation}
where $\dot{W} = (\dot{W}(x, t))_{t \ge 0, x \in \mathbb{R}}$ is a $\mathbb{R}^d$-valued space-time white noise with independent components and $(u_t(x))_{t \ge 0, x \in \mathbb{R}}$ is a continuous $\mathbb{R}^d$-valued process. We also suppose that the noise is adapted with respect to a filtered probability
space $(\Omega, \mathcal{F}, (\mathcal{F}_t), P)$, where $\mathcal{F}$ is complete and $(\mathcal{F}_t)$ is right continuous, in
that $W(f)$ is $\mathcal{F}_t$-measurable whenever $f$ is supported in $[0, t] \times \mathbb{R}$.

Denote $G_t(x) = (4\pi t)^{-1/2}\exp(-x^2/4t)$ as the fundamental solution of the heat equation. The stationary pinned string $(U_t(x))_{ t \ge 0, x \in \mathbb{R}}$, which we will study in this paper, is a solution to \eqref{eq:FunakiModel} driven by the white noise $\dot{W}(x, t)$ such that
\begin{itemize}
    \item $U_0(x) = \int_0^{\infty}\int (G_r(x-z) - G_r(z)) \tilde{W}(dzdr),$ where $\tilde{W}$ is a space-time white noise independent of $\dot{W}$;
    \item $U_t(x)$ is a continuous version of the process $\int G_t(x - z)U_0(z)dz + \int_0^t\int G_r(x-z)W(dzdr)$.
\end{itemize}

Here, we write $f \lesssim g$ if there is a constant $C > 0$ such that $f(x) \le Cg(x)$ for all $x$, and $f \simeq g$ if there are constants $C_1, C_2 > 0$ such that $C_1f(x) \le g(x) \le C_2f(x)$ for all $x$. We also denote $B_{\delta}(z)$ as the box $\{y \in \mathbb{R}^d: |y_i - z_i| < \delta \ \forall i \}$

Before proceeding any further, we restate a few properties of the stationary pinned string, all of which can be found in \cite{MT02}.
\begin{enumerate}
    \item $U_t(x) = \left(U_t^{(1)}(x), \dots, U_t^{(d)}(x) \right)$, where the $U^{(i)}(x)$ are i.i.d. and $(U_0^{(i)}(x))_{ x \in \mathbb{R}}$ is a two-sided Brownian motion with $U_0(0) = 0$.
    \item Each of the $\left(U_t^{(i)}(x)\right)_{t \ge 0, x \in \mathbb{R}}$ are centered Gaussian fields such that
    \begin{align}
        E&\left[\left(U_t^{(i)}(x) - U_t^{(i)}(y) \right)^2 \right] = |x-y| \ \forall x, y \in \mathbb{R}, t \ge 0, 
        \intertext{and for $x, y \in \mathbb{R}, 0 \le s < t$,} 
        E&\left[\left(U_t^{(i)}(x) - U_s^{(i)}(y) \right)^2 \right] = (t-s)^{1/2}F\left(|x-y|(t-s)^{-1/2} \right), \label{eq:ExpressVariance}
        \intertext{where $F: \mathbb{R} \rightarrow \mathbb{R}$ is smooth, bounded below by $(2\pi)^{-1/2}$, and}
        &\lim_{|x| \rightarrow \infty} F(x)/|x| = 1. \nnend 
        \intertext{Moreover, there exists $c_1 > 0$ such that}
        &c_1\left(|x-y| + |t-s|^{1/2} \right) \le  E\left[\left(U_t^{(i)}(x) - U_s^{(i)}(y) \right)^2 \right] \le 2\left(|x-y| + |t-s|^{1/2} \right). \label{eq:BoundVariance}
    \end{align}
    \item (Translation invariance) For any $t_0 \ge 0, x_0 \in \mathbb{R}$, the field 
    \begin{equation*}
        \left(U_{t_0+t}(x_0 \pm x) - U_{t_0}(x_0) \right)_{x \in \mathbb{R}, t \ge 0}
    \end{equation*}
    has the same law as the stationary pinned string.
    \item (Scaling) For $L > 0$, the field 
    \begin{equation*}
        \left(L^{-1}U_{L^{4}t}(L^2x)\right)_{x \in \mathbb{R}, t \ge 0}
    \end{equation*}
    has the same law as the stationary pinned string.
\end{enumerate}
\eqref{eq:BoundVariance} gives us a useful bound as follows.
\begin{proposition}
\begin{equation*}\label{prop:trivialBound}
    P\left(U_t(x) \in B_{\delta}(0) \right) \lesssim \left(\dfrac{\delta}{\left(t^{1/2} + |x| \right)^{1/2}} \right)^d.
\end{equation*}
\end{proposition}
\begin{proof}
\begin{align*}
    P\left(U_t(x) \in B_{\delta}(0) \right) &= P\left(U_t^{(i)}(x) \in B_{\delta}(0) \right)^d \\
    &= \left(\dfrac{1}{\sqrt{2\pi \text{Var}U_t^{(i)}(x)}}\int_{-\delta}^{\delta} \exp\left(-\frac{z^2}{2\text{Var}U_t^{(i)}(x)} \right) dz \right)^{d} \\
    &\le \left(\dfrac{2\delta}{\sqrt{2\pi c_1(t^{1/2} + |x|)}} \right)^{d}.
\end{align*}
\end{proof}

We call $(U_t(x))_{t \ge 0, x \in \mathbb{R}}$
\begin{itemize}
    \item {\bf point recurrent} if, almost surely, $\forall z \in \mathbb{R}^d$, there exist (random) sequences $\{x_n\}, \{t_n\}$ such that $t_n \nearrow \infty$ and $U_{t_n}(x_n) = z \ \forall n$;
    \item {\bf neighborhood recurrent} if, almost surely, $\forall z \in \mathbb{R}^d$ and $\epsilon > 0$, there exist (random) sequences $\{x_n\}, \{t_n\}$ such that $t_n \nearrow \infty$ and $U_{t_n}(x_n) \in B_{\epsilon}(z) \ \forall n$; 
    \item {\bf transient} if, $\inf_{x \in \mathbb{R}}|U_t(x)|$ goes to infinity.
\end{itemize}
In \cite{MT02}, Mueller and Tribe showed that 
\begin{itemize}
    \item for $d \le 5$, $(U_t(x))_{t \ge 0, x \in \mathbb{R}}$ is point recurrent;
    \item for $d = 6$, $(U_t(x))_{t \ge 0, x \in \mathbb{R}}$ is neighborhood recurrent but not point recurrent;
    \item for $d \ge 7$, $(U_t(x))_{t \ge 0, x \in \mathbb{R}}$ is transient.
\end{itemize}

This motivates the question of finding the rate of escape of this string when $d \ge 6$. How fast should a ball centered at $z$ grow so that (neighborhood) recurrence happens for $d \ge 7$? Likewise, how fast should a ball centered at $z$ shrink so that transience happens for $d = 6$?

In this paper, we wish to study the question of recurrence and transience when the growth rate is of the form $f(t) = t^{\alpha}$, where $\alpha \in \mathbb{R}$. Since it suffices to consider the ball centered at the origin, we establish our main results in the following theorems.

\begin{theorem}\label{thm:FirstBigTheorem}
Suppose $(U_t(x))_{t \ge 0, x \in \mathbb{R}}$ is the stationary pinned string in $\mathbb{R}^d \ (d \ge 7)$. Then almost surely,
\begin{equation*}
    \liminf_{t \rightarrow \infty} \dfrac{\inf_{x \in \mathbb{R}}|U_t(x)|}{t^{\alpha}} = 
    \begin{cases} +\infty & (0 < \alpha < 1/4) \\
    0 & (\alpha \ge 1/4).
    \end{cases}
\end{equation*}
This means neighborhood recurrence of $U_t(x)$ happens when $\alpha \ge 1/4$, and transience happens otherwise.
\end{theorem}

\begin{theorem}\label{thm:SecBigTheorem}
Suppose $(U_t(x))_{t \ge 0, x \in \mathbb{R}}$ is the stationary pinned string in $\mathbb{R}^6$. Then almost surely,
\begin{equation*}
     \liminf_{t \rightarrow \infty} \dfrac{\inf_{x \in \mathbb{R}}|U_t(x)|}{t^{-\alpha}} = 0 \ \forall \alpha > 0.
\end{equation*}
This means neighborhood recurrence of $U_t(x)$ happens for all $\alpha > 0$.
\end{theorem}

The remainder of this paper is organized as follows. In Sections \ref{sec:atleast.25} and \ref{sec:lessthan.25}, we prove separate cases of Theorem \ref{thm:FirstBigTheorem}. We then prove Theorem \ref{thm:SecBigTheorem} in \ref{sec:ifdis6}. Finally, we discuss previous results, open questions, and conjectures in \ref{sec:openquest}.  

\section{Proof of Theorem \ref{thm:FirstBigTheorem} when $\alpha \ge 1/4$}\label{sec:atleast.25}
Define $E_{\epsilon}$ be the event that there exist sequences $\{x_n\}, \{t_n\}$ with $t_n \nearrow \infty$ such that $U_{t_n}(x_n) \in B_{\epsilon t_n^{1/4}}(0)$, i.e. $|U_{t_n}(x_n)|/t_n^{1/4} < \epsilon$. Define
\begin{align*}
    \mathcal{G}_N = \sigma&\{U_0(x): |x| > N \} \nnend
    &\vee \sigma\{W(\phi): \phi(t, x) = 0 \text{ if } 0 \le t \le N \text{ and } |x| \le N \}, \nnend 
    \intertext{and}
    \mathcal{G} &= \bigcap_{N \ge 1} \mathcal{G}_N,
\end{align*}
Since $U_0$ and $W$ are independent, using the same arguments to prove Kolmogorov's 0-1 law for the Brownian tail $\sigma$-field, we can show that $\mathcal{G}$ is trivial.
\begin{lemma}\label{lem:tailevent}
$E_{\epsilon}$ belongs to the tail $\sigma$-field $\mathcal{G}$.
\end{lemma}
\begin{proof}
The proof of this lemma is similar to that of {\bf Lemma 5} in \cite{MT02}.
\end{proof}
\begin{proof}[Proof of Theorem \ref{thm:FirstBigTheorem} when $\alpha \ge 1/4$] It suffices to show the result for $\alpha = 1/4$. For any $\epsilon > 0$, 
\begin{align*} 
    P\left(E_{\epsilon} \right) &\ge P\big(\inf_{x \in\mathbb{R}} |U_n(x)|/n^{1/4} < \epsilon \text{ for infinitely many } n \in \mathbb{Z}^+\big). \nnend
    \intertext{By Fatou's Lemma,}
    P\left(E_{\epsilon}\right) &\ge \limsup_{n \rightarrow \infty} P\left(\inf_{x \in\mathbb{R}} |U_n(x)|/n^{1/4} < \infty \right). \nnend
    \intertext{By scaling,}
    P\left(E_{\epsilon} \right) &\ge \limsup_{n \rightarrow \infty} P\left(\inf_{x \in \mathbb{R}} \left|U_1(x/\sqrt{n})\right| < \epsilon\right) \nnend 
    &= \limsup_{n \rightarrow \infty} P\left(\inf_{x \in \mathbb{R}} \left|U_1(x)\right| < \epsilon\right) \nnend
    &= P\left(\inf_{x \in \mathbb{R}} |U_1(x)| < \epsilon\right) \nnend
    &\ge P\left(|U_1(0)| < \epsilon \right) \nnend
    &= P\left(|U_1^{(1)}(0)| < \epsilon \right)^d > 0,
\end{align*}
the last inequality follows since $U_1^{(1)}(0)$ is a non-degenerate centered Gaussian random variable. \newline
From Lemma \ref{lem:tailevent}, since $E_{\epsilon}$ is a tail event in $\mathcal{G}$, we have that 
\begin{equation*}
    P(E_{\epsilon}) = 1,
\end{equation*}
which holds for every $\epsilon > 0$, concluding our proof.
\end{proof}

\section{Proof of Theorem \ref{thm:FirstBigTheorem} when $0 < \alpha < 1/4$}\label{sec:lessthan.25}
We start with the following lemma.
\begin{lemma}\label{lem:belowtn}
Define the sequence $\{t_n\}_{n \ge 1}$ as follows.
\begin{equation*}
    \begin{cases}
    t_1 = 1 \\
    t_{n+1} = t_n + t_n^{4\alpha} \ (n \ge 1).
    \end{cases}
\end{equation*}
Then $t_n = \Omega\left(n^{1/(1-4\alpha)}\right)$, i.e. there exists a constant $c = c(\alpha) > 0$ such that 
\begin{equation}\label{eq:boundtnBelow}
    t_n \ge c n^{1/(1-4\alpha)}
\end{equation} for all $n \ge 1$.
\end{lemma}

\begin{proof}
Denote $r = 1/(1-4\alpha) > 1$. We show by induction that \eqref{eq:boundtnBelow} holds for $c = (1/2)^{\lceil r \rceil r} \in (0, 1)$, where $\lceil \cdot \rceil$ is the ceiling function. 

\eqref{eq:boundtnBelow} trivially holds for $n = 1$. Suppose it holds for $n = k \ge 1$, i.e. $t_k \ge ck^r$. 
We see that
\begin{equation*}
    \left(1 + \dfrac{1}{k} \right)^r \le \left(1 + \dfrac{1}{k} \right)^{\lceil r \rceil} = 1 + \sum_{i = 1}^{\lceil r \rceil} {\lceil r \rceil \choose i}\dfrac{1}{k^i} \le 1 + \dfrac{1}{k}\sum_{i = 0}^{\lceil r \rceil}{\lceil r \rceil \choose i} = 1 + \dfrac{2^{\lceil r \rceil}}{k} = 1 + \dfrac{c^{-1/r}}{k}.
\end{equation*}
This implies
\begin{equation}\label{eq:derboundtn}
    c(1+k)^r \le ck^r + c^{1 - 1/r}k^{r-1}.
\end{equation}
By the induction hypothesis,
\begin{equation}\label{eq:termboundtn}
    t_{k+1} = t_k + t_k^{4\alpha} \ge ck^r + \left(ck^r\right)^{1 - 1/r}.
\end{equation}
From \eqref{eq:derboundtn} and \eqref{eq:termboundtn}, we see that
\begin{equation*}
    t_{k+1} \ge c(1+k)^r,
\end{equation*}
which completes the proof of Lemma \ref{lem:belowtn}.
\end{proof}

\begin{proof}[Proof of Theorem \ref{thm:FirstBigTheorem} when $0 < \alpha < 1/4$]
Our strategy closely follows that found in Theorem 3 of \cite{MT02}. We find a grid of points and show that recurrence of this string along this grid is impossible, then control the regions between these grid points.  

Define the sequence $\{t_n\}_{n \ge 1}$ as in Lemma \ref{lem:belowtn}. On the lines $t = t_{n+1}$, choose points $kt_n^{2\alpha} \ (k \in \mathbb{Z})$. Let $R_{n, k}$ be the rectangles with vertices $(t_n, kt_n^{2\alpha})$, $(t_{n+1}, t_n^{2\alpha})$, $(t_n, (k+1)t_n^{2\alpha})$, $(t_{n+1}, (k+1)t_n^{2\alpha})$. Define $m(n, k) = \lfloor \left(n^{1/2} + |k|\right)^{d/6 - 1.1} \rfloor$. We divide $R_{n, k}$ into $m(n, k)^3$ rectangles, each of them a translate of $\left[0, t_n^{4\alpha}m(n, k)^{-2} \right] \times \left[0, t_n^{2\alpha}m(n, k)^{-1}\right]$. For these rectangles, the points with the largest $(x, t)$ coordinates are $\left(t_n + it_n^{4\alpha}m(n, k)^{-2}, kt^{2\alpha} + jt_n^{2\alpha}m(n, k)^{-1}\right)$ $=: (t_{(n, k, i)}, x_{(n, k, j)})$ $(1 \le i \le m(n, k)^2; 1 \le j \le m(n, k))$; these will be our grid points. 

\subsection{Transience of $U_t(x)$ at grid points}\label{subsec:gridpoint}
By Proposition \ref{prop:trivialBound},
\begin{align}\label{eq:BorCanGrid}
    \sum_{n \ge 1}&\sum_{k \in \mathbb{Z}}\sum_{i=1}^{m(n, k)^2}\sum_{j=1}^{m(n, k)}P\left(\left|U_{t_{(n, k, i)}}(x_{(n, k, j)})
    \right| \le B_{2\delta t_{(n, k, i)}^{\alpha}}(0)\right) \nnend 
    &\lesssim \sum_{n \ge 1}\sum_{k \in \mathbb{Z}}\sum_{i=1}^{m(n, k)^2}\sum_{j=1}^{m(n, k)} \left(\dfrac{\delta{t_{(n, k, i)}}^{\alpha}}{\left(t_{(n, k, i))}^{1/2} + |x_{(n, k, j)}| \right)^{-1/2}} \right)^d \nnend
    &\lesssim \delta^d\sum_{n \ge 1}\sum_{k \ge 0}\sum_{i=1}^{m(n, k)^2}\sum_{j=1}^{m(n, k)} \left(\dfrac{{t_{(n, k, i)}}^{\alpha}}{\left(t_{(n, k, i))}^{1/2} + |x_{(n, k, j)}| \right)^{-1/2}} \right)^d,
\end{align}
Now, 
\begin{equation}\label{eq:EqAllCombi}
    \dfrac{{t_{(n, k, i)}}^{2\alpha}}{t_{(n, k, i))}^{1/2} + |x_{(n, k, j)}|} = \dfrac{1}{t_{(n, k, i))}^{1/2 - 2\alpha} + |x_{(n, k, j)}|{t_{(n, k, i)}}^{-2\alpha}}.
\end{equation}
We bound the terms in the denominator in \eqref{eq:EqAllCombi} as follows. By Lemma \ref{lem:belowtn},
\begin{equation}\label{eq:firstterm}
    t_{(n, k, i))}^{1/2 - 2\alpha} \gtrsim \left(n^{1/(1-4\alpha)}\right)^{1/2 - 2\alpha} = n^{1/2}.
\end{equation}
Also,
\begin{align}\label{eq:secondterm}
    |x_{(n, k, j)}|{t_{(n, k, i)}}^{-2\alpha} &= t_n^{2\alpha}\left(k + jm(n, k)^{-1} \right)\left(t_n + it_n^{4\alpha}m(n, k)^{-2} \right)^{-2\alpha}. \nnend 
    \intertext{Since $1 \le i \le m(n, k)^{-2}$,}
    |x_{(n, k, j)}|{t_{(n, k, i)}}^{-2\alpha} &\ge t_n^{2\alpha}\left(k + jm(n, k)^{-1} \right)\left(t_n + t_n^{4\alpha} \right)^{-2\alpha} \nnend
    &= \left(t_n/t_{n+1} \right)^{2\alpha}\left(k + jm(n, k)^{-1}\right). \nnend 
    \intertext{Since $t_n/t_{n+1} \ge 1/2$,}
    |x_{(n, k, j)}|{t_{(n, k, i)}}^{-2\alpha} &\gtrsim k + jm(n, k)^{-1} \ge k. \nnend
\end{align}
From \eqref{eq:EqAllCombi}, \eqref{eq:firstterm}, and \eqref{eq:secondterm},
\begin{equation}\label{eq:eachterm}
    \dfrac{{t_{(n, k, i)}}^{2\alpha}}{t_{(n, k, i))}^{1/2} + |x_{(n, k, j)}|} \lesssim \dfrac{1}{n^{1/2} + k}.
\end{equation}
From \eqref{eq:BorCanGrid} and \eqref{eq:eachterm},
\begin{align}\label{eq:BorCanGridOne}
   \sum_{n \ge 1}\sum_{k \in \mathbb{Z}}&\sum_{i=1}^{m(n, k)^2}\sum_{j=1}^{m(n, k)}P\left(\left|U_{t_{(n, k, i)}}(x_{(n, k, j)})
    \right| \le B_{2\delta t_{(n, k, i)}^{\alpha}}(0)\right) \nnend
    &\lesssim  \delta^d\sum_{n \ge 1}\sum_{k \ge 0}\sum_{i=1}^{m(n, k)^2}\sum_{j=1}^{m(n, k)} \left(n^{1/2} + k\right)^{-d/2} \nnend
    &= \delta^d\sum_{n \ge 1}\sum_{k \ge 0} m(n, k)^3\left(n^{1/2} + k\right)^{-d/2} \nnend 
    &\le \delta^d\sum_{n \ge 1}\sum_{k \ge 0}\left(n^{1/2} + k\right)^{d/2 - 3.3}\left(n^{1/2} + k\right)^{-d/2} \nnend 
    &= \delta^d\sum_{n \ge 1}\sum_{k \ge 0}\left(n^{1/2} + k\right)^{- 3.3} < \infty,
\end{align}
the last inequality follows from the integral test
\begin{equation*}
    \int_{1}^{\infty}\int_{0}^{\infty}\dfrac{dydx}{\left(x^{1/2} + y \right)^{3.3}} < \infty.
\end{equation*}
By the Borel-Cantelli lemma, the string $U_t(x)$, evaluated at these grid points, will eventually leave their corresponding boxes $B_{2\delta t_{(n, k, i)}^{\alpha}}(0)$ for large (random) $t$. 

\subsection{Controlling the regions between grid points}\label{subsec:betweengridpoints}
From the display after $(6.8)$ in \cite{MT02}, we can find constants $c_1, c_2 > 0$ such that
\begin{equation}\label{eq:boundExp}
    P\left(\sup_{(t, x) \in [0,1]^2} |U_t(x)| \ge \delta \right) \le c_1 \exp\left(-c_2\delta^2 \right).
\end{equation}
Denote $R_{(n, k, i, j)}$ as the translation of $\left[0, t_n^{4\alpha}m(n, k)^{-2} \right] \times \left[0, t_n^{2\alpha}m(n, k)^{-1}\right]$ with largest coordinates $(t_{(n, k, i)}, x_{(n, k, j)})$. Then by translation,
\begin{align*}
    \sum_{n \ge 1}\sum_{k \in \mathbb{Z}}&\sum_{i=1}^{m(n, k)^2}\sum_{j=1}^{m(n, k)}P \left( \sup_{(t, x) \in R_{(n, k, i, j)}} \left|U_t(x) - U_{t_{(n, k, i)}}(x_{(n, k, j)}) \right| \ge \delta t_{(n, k, i)}^{\alpha} \right) \nnend 
    = \sum_{n \ge 1}\sum_{k \in \mathbb{Z}}&\sum_{i=1}^{m(n, k)^2}\sum_{j=1}^{m(n, k)}P \left( \sup_{(t, x) \in \left[0, t_n^{4\alpha}m(n, k)^{-2} \right] \times \left[0, t_n^{2\alpha}m(n, k)^{-1}\right]} \left|U_t(x)\right| \ge \delta t_{(n, k, i)}^{\alpha} \right). \nnend 
    \intertext{By the scaling of $U_t(x)$, the preceding quadruple sum becomes}
    &= \sum_{n \ge 1}\sum_{k \in \mathbb{Z}}\sum_{i=1}^{m(n, k)^2}\sum_{j=1}^{m(n, k)}P \left( \sup_{(t, x) \in [0, 1]^2} \left|U_t(x)\right| \ge \delta t_{(n, k, i)}^{\alpha} m(n, k)^{1/2}t_n^{-\alpha} \right). \nnend 
    \intertext{By \eqref{eq:boundExp}, we can bound the sum above by}
    &\lesssim \sum_{n \ge 1}\sum_{k \in \mathbb{Z}}\sum_{i=1}^{m(n, k)^2}\sum_{j=1}^{m(n, k)} \exp \left( -c_2 \delta^2 m(n, k)(t_{(n, k, i)} / t_n)^{2\alpha} \right) \nnend
    &= \sum_{n \ge 1}\sum_{k \in \mathbb{Z}}\sum_{i=1}^{m(n, k)^2}\sum_{j=1}^{m(n, k)} \exp \left( -c_2 \delta^2 m(n, k)\left(\left(t_n + it_n^{4\alpha}m(n, k)^{-2} \right) / t_n\right)^{2\alpha} \right) \nnend 
    &\le \sum_{n \ge 1}\sum_{k \in \mathbb{Z}}\sum_{i=1}^{m(n, k)^2}\sum_{j=1}^{m(n, k)} \exp \left( -c_2 \delta^2 m(n, k) \right) \nnend 
    &= \sum_{n \ge 1}\sum_{k \in \mathbb{Z}} m(n, k)^3 \exp \left( -c_2 \delta^2 m(n, k) \right) \nnend 
    &\lesssim \sum_{n \ge 1}\sum_{k \ge 0} m(n, k)^3 \exp \left( -c_2 \delta^2 m(n, k) \right) \nnend 
    &\simeq \sum_{n \ge 1}\sum_{k \ge 0} \left(n^{1/2} + k \right)^{d/2 - 3.3} \exp \left( -c_2 \delta^2 \left(n^{1/2} + k \right)^{d/6-1.1} \right) < \infty,
\end{align*}
the last inequality can be shown using the integral test for convergence of series. 

By the Borel-Cantelli lemma, there exists a (random) $N_0 \in \mathbb{Z}^+$ such that for all $n \ge N_0$, $k \in \mathbb{Z}$, $i \le m(n, k)^2$, and $j \le m(n, k)$, we have that
\begin{equation}\label{eq:explainStuff}
    \sup_{(t, x) \in R_{(n, k, i, j)}}\left|U_t(x) - U_{t_{(n, k, i)}}(x_{(n,k,j)}) \right| < \delta t_{(n, k, i)}^{\alpha}.
\end{equation}
If $U_t(x)$ evaluated at the grid point $(t_{(n, k,i)}, x_{(n, k ,j)})$ is outside of the box $B_{2\delta t_{(n, k, i)}^{\alpha}}(0)$ and \eqref{eq:explainStuff} holds, then none of the values $U_t(x)$, where $(t, x) \in R_{(n, k, i, j)}$, can be within $\delta t^{\alpha} < \delta t_{(n, k, i)}^{\alpha}$ of $0$.

Combining Subsections \ref{subsec:gridpoint} and \ref{subsec:betweengridpoints}, we see that the probability of recurrence is zero, thus completing the proof.
\end{proof}

\section{A result when $d = 6$}\label{sec:ifdis6}
We start with the following lemmas.

\begin{lemma}\label{lem:boundForTwoPoints}
Let $\delta \in (0, 1)$ and $\alpha > 0$. Then
\begin{equation}\label{eq:boundForTwoPoints}
    P\left(U_t(x) \in B_{\delta t^{-\alpha}}(0), U_{t+s}(x+y) \in B_{\delta (t+s)^{-\alpha}}(0) \right) \lesssim C \dfrac{(t+s)^{-6\alpha} t^{-6\alpha}}{\left(s^{1/2} + |y| \right)^3\left(t^{1/2} + |x| \right)^3}
\end{equation}
holds under one of the following conditions:
\begin{enumerate}
    \item $t \ge 1, |x|, |y| \le 2t^{1/2},$ and $0 \le s \le t$;
    \item $|y| \le 2s^{1/2}$,
\end{enumerate}
where $C = C(\delta) > 0$ is dependent only on $\delta$.
\end{lemma}
The proof of this lemma is deferred to Appendix \ref{sec:proofboundForTwoPoints}. 

\begin{lemma}\label{lem:growthInPower}
For any $\alpha \in \mathbb{Z}^+$, there exists a constant $C = C(\alpha) > 0$ such that 
\begin{equation*}
    \left(\left(1 + (N^2 - N)^{-(6\alpha+1)} \right)^{1/(6\alpha + 1)} -1\right)^{-1} \lesssim N^C. 
\end{equation*}
\end{lemma}
\begin{proof}
Let $K = 6\alpha + 1 \in \mathbb{Z}^+$, and let $C = 2K + 2$. 
Then for any $\rho > 0$,
\begin{align*}
(N^C + \rho)^K &= N^{CK} + \sum_{i=1}^K {K \choose i} N^{C(K-i)}\rho^i \nnend 
&\le N^{CK} + N^{C(K-1)}(1 + \rho)^{K} \nnend 
&\le N^{CK} + N^{(C-2)K} N^{-2} (1 + \rho)^K. \nnend 
\intertext{Choosing $\rho = 4^{1/K} - 1 > 0$, then for all $N \ge 2$,}
(N^C + \rho)^K & \le N^{CK} + N^{(C-2)K}.
\intertext{Thus,}
\rho & \le N^C \left(1 + N^{-2K} \right)^{1/K} - N^C \nnend 
&\le \left(\left(1 + (N^2 - N)^{-K} \right)^{1/K} -1\right)N^C,
\end{align*}
completing our proof of the lemma.
\end{proof}

\begin{lemma}{(An inclusion-exclusion type lower bound)}\label{lem:IncluExclu} Let $\{A_i\}_{1 \le i \le n}$ be events and $A = \bigcup_{i=1}^n A_i$. Then
\begin{equation}
    P(A) \ge \dfrac{\left(\sum_{i=1}^n P(A_i)\right)^2}{\sum_{i=1}^n P(A_i) + 2 \sum_{1\le i < j \le n}P(A_i \cap A_j)}.
\end{equation}
\end{lemma}
We omit the proof for this standard lemma.
\begin{proof}[Proof of Theorem \ref{thm:SecBigTheorem}] We mimic the same strategy as that in the proof of Theorem $3$ of \cite{MT02}. It suffices to show for $\alpha \in \mathbb{Z}^+$. Here, we fix $\delta \in (0, 1)$ (thus, constants that are only dependent on $\delta$ and/or $\alpha$ are treated as absolute constants). Let $\mathcal{R}(\delta)$ be the event that there exist sequences $\{x_n\}, \{t_n\}$ with $t_n \nearrow \infty$ such that $U_{t_n}(x_n) \in B_{\delta t_n^{-\alpha}}(0)$. As in Lemma \ref{lem:tailevent}, we can show that $\mathcal{R}(\delta)$ is an event in the tail $\sigma$-field $\mathcal{G}$, where $\mathcal{G}$ is defined in Section \ref{sec:atleast.25}.

Denote $k = 1/(6\alpha + 1) \in (0, 1)$. For integers $i, j, N$, define 
\begin{align*}
    \mathcal{R}_{i, j}^{(N)} &= \left\{U_{N + i^k}(j) \in B_{\delta(N + i^k)^{-\alpha}}(0)\right\}, \\
    \mathcal{R}(N, \delta) &= \bigcup_{i: N \le N + i^k \le N^2} \bigcup_{0 \le j \le (N+i^k)^{1/2}} \mathcal{R}_{i, j}^{(N)}. \\
\end{align*}
Using Lemma \ref{lem:IncluExclu}, we show that there exists $p_0 > 0$ such that $P(\mathcal{R}(N, \delta)) \ge p_0 > 0$ for all sufficiently large $N$. Then, since 
\begin{equation*}
    P(\mathcal{R}(\delta)) \ge P(\mathcal{R}(N, \delta) \text{ infinitely often}) \ge p_0 > 0,
\end{equation*}
we get $P(\mathcal{R}(\delta)) = 1$ by the zero-one law for any $\delta > 0$, concluding our proof. To bound $P(\mathcal{R}(N, \delta))$ below, we find bounds for the sum of $P\left(\mathcal{R}_{i, j}^{(N)} \right)$ (which we will call the "first-order term"), and the sum of $P\left(\mathcal{R}_{i, j}^{(N)} \cap \mathcal{R}_{i', j'}^{(N)} \right)$ (which we will call the "covariance term").

\subsection{Bounding the first-order term}\label{subsec:firstOrderTerm}
Denote $r = \lfloor (N^2 - N)^{1/k} \rfloor$. Using the variance estimate in $\eqref{eq:BoundVariance}$, similar to the proof of Proposition \ref{prop:trivialBound}, we get
\begin{equation}
    P\left(\mathcal{R}_{i, j}^{(N)} \right) \simeq \left(N+i^k\right)^{-6\alpha}\left(\left(N + i^k \right)^{1/2} + |j|\right)^{-3},
\end{equation}
where $\simeq$ is defined in Section \ref{sec:intro}. Thus, 
\begin{align}
    &\sum_{0 \le i \le (N^2-N)^{1/k}}\sum_{0 \le j \le (N+i^k)^{1/2}}P\left(\mathcal{R}_{i, j}^{(N)} \right) \nnend 
    \simeq &\sum_{i=0}^r\sum_{0 \le j \le (N+i^k)^{1/2}}\left(N+i^k\right)^{-6\alpha}\left(\left(N + i^k \right)^{1/2} + |j|\right)^{-3} \nnend 
    \simeq &\int_0^{r} \int_0^{(N+i^k)^{1/2}}\left(N+x^k\right)^{-6\alpha}\left(\left(N + x^k \right)^{1/2} + |y|\right)^{-3} dydx \nnend 
    \simeq &\int_0^{r} \left(N+x^k\right)^{-6\alpha-1}dx \nnend 
    \intertext{Setting $z = N + x^k$, then the above is equal to}
    = &\int_N^{N^2} (6\alpha + 1) (z-N)^{6\alpha} z^{-6\alpha - 1} dz \nnend
    \simeq &\int_N^{N^2} (1 - N/z)^{6\alpha} z^{-1}dz. \label{eq:orderOfSimpleSum}
\end{align}
Trivially,
\begin{equation}\label{eq:orderOfSimpleSumAbove}
    \int_N^{N^2} (1 - N/z)^{6\alpha} z^{-1}dz \le \int_N^{N^2} z^{-1}dz = \log N.
\end{equation}
For $N$ large enough,
\begin{equation}\label{eq:orderOfSimpleSumBelow}
     \int_N^{N^2} (1 - N/z)^{6\alpha} z^{-1}dz \ge  \int_{2N}^{N^2} (1 - N/z)^{6\alpha} z^{-1}dz \ge \int_{2N}^{N^2} (1/2)^{6\alpha} z^{-1}dz \simeq \log N.
\end{equation}
From \eqref{eq:orderOfSimpleSum}, \eqref{eq:orderOfSimpleSumAbove}, and \eqref{eq:orderOfSimpleSumBelow}, 
\begin{equation}\label{eq:finalOrderOfSimpleSum}
    \sum_{i=0}^r\sum_{0 \le j \le (N+i^k)^{1/2}}P\left(\mathcal{R}_{i, j}^{(N)} \right) \simeq \log N.
\end{equation}

\subsection{Bounding the covariance term}
The covariance term is as follows.
\begin{align}
    &\sum_{i=0}^r\sum_{0 \le j \le (N+i^k)^{1/2}}\sum_{i'=0}^r\sum_{0 \le j' \le (N+i^k)^{1/2}}P\left(\mathcal{R}_{i, j}^{(N)} \cap \mathcal{R}_{i', j'}^{(N)} \right)1_{\{(i, j) \neq (i', j')\}} \nnend 
    \simeq &\sum_{i=0}^r\sum_{0 \le j \le (N+i^k)^{1/2}}\sum_{i'=i}^r\sum_{0 \le j+j' \le (N+i^k)^{1/2}}P\left(\mathcal{R}_{i, j}^{(N)} \cap \mathcal{R}_{i', j+j'}^{(N)} \right)1_{\{(i', j') \neq (i, 0)\}} .\nnend 
    \intertext{Since $-\left(N+i'^k\right)^{1/2} \le -\left(N+i^k\right)^{1/2} \le -j \le j' \le \left(N+i^k\right)^{1/2} -j \le \left(N+i'^k\right)^{1/2}$, the quadruple sum above is at most}
    \le &\sum_{i=0}^r\sum_{0 \le j \le (N+i^k)^{1/2}}\sum_{i'=i}^r\sum_{|j'| \le (N+i'^k)^{1/2}}P\left(\mathcal{R}_{i, j}^{(N)} \cap \mathcal{R}_{i', j+j'}^{(N)} \right)1_{\{(i', j') \neq (i, 0)\}}. \label{eq:CovTermBound}
\end{align}
Setting $t = N + i^k$, $s = i'^k - i^k$, $x = j$, $y = j'$, we see that $t \ge 1$, $s \ge 0$, $0 \le x \le t^{1/2}$, and $|y| \le (t+s)^{1/2}$. Consider the following cases.
\begin{enumerate}
    \item If $s \le t$, then $x \le t^{1/2} < 2t^{1/2}$ and $|y| \le (t+s)^{1/2} < 2t^{1/2}$;
    \item If $s > t$, then $|y| \le (t+s)^{1/2} < 2s^{1/2}$.
\end{enumerate}
In any case, the conditions in Lemma \ref{lem:boundForTwoPoints} hold. Thus, 
\begin{equation}\label{eq:boundCovarSummand}
    P\left(\mathcal{R}_{i, j}^{(N)} \cap \mathcal{R}_{i', j+j'}^{(N)} \right) \lesssim \dfrac{(N + i'^k)^{-6\alpha} (N+i^k)^{-6\alpha}}{\left((i'^k - i^k)^{1/2} + |j'| \right)^3\left((N+i^k)^{1/2} + j \right)^3}.
\end{equation}
We split the quadruple sum in \eqref{eq:CovTermBound} into two parts: $i' > i$, and $i' = i$. In the first case, $j' = 0$ is included in the summation, whereas it is not in the second case (since $(i', j') \neq (i, 0)$).
For the first case, using \eqref{eq:boundCovarSummand},
\begin{align}
    &\sum_{i=0}^r\sum_{0 \le j \le (N+i^k)^{1/2}}\sum_{i'=i+1}^r\sum_{|j'| \le (N+i'^k)^{1/2}}P\left(\mathcal{R}_{i, j}^{(N)} \cap \mathcal{R}_{i', j+j'}^{(N)} \right) \nnend 
    \lesssim &\sum_{i=0}^r\sum_{0 \le j \le (N+i^k)^{1/2}}\sum_{i'=i+1}^r\sum_{0 \le j' \le (N+i'^k)^{1/2}}\dfrac{(N + i'^k)^{-6\alpha} (N+i^k)^{-6\alpha}}{\left((i'^k - i^k)^{1/2} + |j'| \right)^3\left((N+i^k)^{1/2} + j \right)^3} \nnend 
    \simeq &\int_0^r \int_0^{(N+x^k)^{1/2}} \int_{x+1}^r \int_0^{(N+x'^k)^{1/2}}\dfrac{(N + x'^k)^{-6\alpha} (N+x^k)^{-6\alpha}}{\left((x'^k - x^k)^{1/2} + y' \right)^3\left((N+x^k)^{1/2} + y \right)^3}dy'dx'dydx \nnend 
    \lesssim &\int_0^r \int_0^{(N+x^k)^{1/2}} \int_{x+1}^r \dfrac{(N + x'^k)^{-6\alpha} (N+x^k)^{-6\alpha}}{\left(x'^k - x^k \right)\left((N+x^k)^{1/2} + y \right)^3}dx'dydx \nnend 
    \simeq &\int_0^r \int_{x+1}^r \dfrac{(N + x'^k)^{-(6\alpha+1)} (N+x^k)^{-6\alpha}}{x'^k - x^k}dx'dx \nnend 
    \intertext{Setting $z = N + x^k$, $z' = N + x'^k$, the above is equal to}
    = &\int_N^{N^2} \int_{((z-N)^{1/k} + 1)^k + N +1}^{N^2} k^{-2}\left(1 - \dfrac{N}{z} \right)^{6\alpha}\left(1 - \dfrac{N}{z'} \right)^{6\alpha}\dfrac{1}{z(z'-z)}dz'dz \nnend 
    \lesssim & \int_N^{N^2} \int_{((z-N)^{1/k} + 1)^k + N+1}^{N^2} \dfrac{dz'dz}{z(z'-z)} \nnend 
    = & \int_N^{N^2} z^{-1} \left(\log(N^2 - z) - \log(((z-N)^{1/k} + 1)^k - (z - N) \right)dz \nnend 
    \le & \left(\log(N^2 - N) - \log(((N^2-N)^{1/k} + 1)^k - (N^2 - N) \right)\int_N^{N^2}z^{-1}dz \nnend 
    = & \log  \left(\left(1 + (N^2 - N)^{-(6\alpha+1)} \right)^{1/(6\alpha + 1)} -1\right)^{-1} \log N \nnend
    \intertext{By Lemma \ref{lem:growthInPower}, for some constant $C > 0$, the above is at most}
    \le &\log N^C \log N \lesssim (\log N)^2. \label{eq:boundFirstMinorTerm}
\end{align}
For the second case, using \eqref{eq:boundCovarSummand} again, 
\begin{align}
    &\sum_{i=0}^r\sum_{0 \le j \le (N+i^k)^{1/2}}\sum_{1\le |j'| \le (N+i^k)^{1/2}}P\left(\mathcal{R}_{i, j}^{(N)} \cap \mathcal{R}_{i, j+j'}^{(N)} \right) \nnend 
    \lesssim & \sum_{i=0}^r\sum_{0 \le j \le (N+i^k)^{1/2}}\sum_{1\le |j'| \le (N+i^k)^{1/2}} \dfrac{ (N+i^k)^{-12\alpha}}{j'^3\left((N+i^k)^{1/2} + j \right)^3} \nnend 
    \simeq & \int_0^{r}\int_0^{(N+x^k)^{1/2}}\int_{1}^{(N+x^k)^{1/2}} \dfrac{ (N+x^k)^{-12\alpha}}{y'^3\left((N+x^k)^{1/2} + y \right)^3} \nnend 
    \simeq &\int_0^r (N+x^k)^{-12\alpha - 1}\left(1 - (N+x^k)^{-1} \right)dx \nnend 
    < & \int_0^r (N+x^k)^{-12\alpha - 1}dx \nnend 
    \intertext{Setting $z = N + x^k$, the above is equal to}
    = &\int_N^{N^2}k^{-1}(z-N)^{6\alpha}z^{-12\alpha -1}dz \ll \log N. \label{eq:boundSecondMinorTerm}
\end{align}
From \eqref{eq:CovTermBound}, \eqref{eq:boundFirstMinorTerm}, and \eqref{eq:boundSecondMinorTerm}, 
\begin{equation}\label{eq:FinalCovarianceBound}
    \sum_{i=0}^r\sum_{0 \le j \le (N+i^k)^{1/2}}\sum_{i'=0}^r\sum_{0 \le j' \le (N+i^k)^{1/2}}P\left(\mathcal{R}_{i, j}^{(N)} \cap \mathcal{R}_{i', j'}^{(N)} \right)1_{\{(i, j) \neq (i', j')\}} \lesssim (\log N)^2. 
\end{equation}
Using the bounds in \eqref{eq:finalOrderOfSimpleSum} and \eqref{eq:FinalCovarianceBound} from the Subsections above, applying Lemma \ref{lem:IncluExclu} for $\mathcal{R}(N, \delta) = \bigcup_{i=0}^r \bigcup_{0 \le j \le (N+i^k)^{1/2}} \mathcal{R}_{i, j}^{(N)}$, we get 
\begin{equation*}
    P\left(\mathcal{R}(N, \delta) \right) \gtrsim \dfrac{1}{(\log N)^{-1} + 1} \gtrsim 1,  
\end{equation*}
i.e. there exists a $p_0 > 0$ such that $P\left(\mathcal{R}(N, \delta) \right) \ge p_0$ for sufficiently large $N$. Our proof is complete.
\end{proof}

\section{Open questions}\label{sec:openquest}
The rate of escape of Brownian motion has been well-studied (see \cite{C82}, \cite{E80}, \cite{KS99}, \cite{MP10}, and \cite{S58}). Recall that $f = \Omega(g)$ if there exists $C > 0$ such that $f(x) \ge Cg(x)$ for sufficiently large $x$. A shortcoming of our result is that for $d \ge 7$, we do not take into account increasing functions that are $o(t^{1/4})$ and $\Omega\left(t^{1/4 - \epsilon}\right)$ (for example, when the growth rate $t^{1/4}/\ln t$). In many of the finite-dimensional Brownian motions, an integral test is usually used to determine a necessary and sufficient condition for recurrence. For example, for Brownian motion $(B(t))_{t \ge 0}$ in $\mathbb{R}^d \ (d \ge 3)$, we have the following result:
\begin{theorem} {\bf (Dvoretzky-Erd\H{o}s test)}
Let $(B(t))_{t \ge 0}$ in $\mathbb{R}^d \ (d \ge 3)$ and $f: \mathbb{R}^+ \rightarrow \mathbb{R}^+$ increasing. Then 
\begin{align*}
    \int_{1}^{\infty} \left(f(t)t^{-1/2} \right)^{d-2}t^{-1} dt < \infty \qquad \text{ if and only if } \qquad \liminf_{t \rightarrow \infty}\dfrac{B(t)}{f(t)} = \infty \text{ a.s.}
\end{align*}
Conversely, if the integral diverges, then $\liminf_{t \rightarrow \infty}\dfrac{B(t)}{f(t)} = 0$ a.s. 
\end{theorem}
In other words, recurrence happens if and only if the integral diverges. Though the stationary pinned string is not Brownian, it is still Gaussian, allowing for possible analogies. From Theorem \ref{thm:FirstBigTheorem} above, we suspect that a similar condition holds for the stationary pinned string $\left(U_t(x)\right)_{t \ge 0, x \in \mathbb{R}}$ when $d \ge 7$.
\begin{conjecture}
For $d \ge 7$, there exists a constant $C = C(d) > 0$ such that the following holds: given $f: \mathbb{R}^+ \rightarrow \mathbb{R}^+$ increasing, then
\begin{align*}
    \int_{1}^{\infty} \left(f(t)t^{-1/4} \right)^{C}t^{-1} dt < \infty \qquad \text{ if and only if } \qquad \liminf_{t \rightarrow \infty}\dfrac{\inf_{x \in \mathbb{R}}|U_t(x)|}{f(t)} = \infty \text{ a.s.}
\end{align*}
Conversely, if the integral diverges, then $\liminf_{t \rightarrow \infty}\dfrac{\inf_{x \in \mathbb{R}}|U_t(x)|}{f(t)} = 0$ a.s. 
\end{conjecture}
The situation is even less well-understood in the critical dimension $d = 6$, where we only managed to bound the shrinking rate of $f(t)$ on one side. Interestingly, the difficulties for $d = 6$ are encountered not only when we study the question of recurrence, but also in hitting problems. Recall that a $\mathbb{R}^d$-valued process $(u_t(\cdot))_{t \ge 0}$ is said to hit the point $z \in \mathbb{R}^d$ if 
\begin{equation*}
    P(u_t(x) = z \text{ for some } t > 0, x \in \mathbb{R}) > 0.
\end{equation*}
We also say that $d_c$ is the critical dimension if hitting of $B = \{z\}$ occurs for $d < d_c$ but not for $d > d_c$. For the nonlinear stochastic heat equation  
\begin{equation*}\label{eq:Funaki}
    \dfrac{\partial u_t(x)}{\partial t} = \dfrac{\partial^2 u_t(x)}{\partial x^2} + \sigma(u_t(x))\dot{W}(t, x),
\end{equation*}
where the white noise in \eqref{eq:FunakiModel} is multiplied by a matrix-valued function with certain restrictions (see \cite{DKE09}, \cite{DMX21}), the critical dimension is known to be $d_c = 6$. Unlike in the case of vector-valued Brownian sheet and other classes of Gaussian fields, where the sets of hitting points are relatively well-understood (see \cite{KS99}, \cite{DMX17}), it is only known that for the nonlinear stochastic heat equation, almost every point in $\mathbb{R}^6$ is not hit. 

Inspired by the case of $N$-parameter $d$-dimensional Brownian sheet (see \cite{KS99}), we suspect that an exponential shrinking rate might suffice for recurrence of the stationary pinned string. However, the tools from potential theory, which was developed in the mentioned paper, is intractable in solving this problem. Again, the ultimate goal is to find an integral test to determine the necessary and sufficient condition for recurrence.

\section{Acknowledgements}
The author would like to thank Carl Mueller for his mentorship and invaluable suggestions.

\appendix 

\section{Proof of Lemma \ref{lem:boundForTwoPoints}}\label{sec:proofboundForTwoPoints}
We start with the following lemma.
\begin{lemma}\label{lem:boundForSmallTime}
When $d = 6$, for all $s, t \in [1, 2]$, $|x|, |y| \le 2$, and $\delta_1, \delta_2 \in (0, 1)$,
\begin{equation}\label{eq:boundForSmallTime}
    P\left(U_t(x) \in B_{\delta_1}(0), U_s(y) \in B_{\delta_2}(0)\right) \lesssim \delta_1^{6}\delta_2^{6}\left(|t-s|^{1/2} + |x-y| \right)^{-3}. 
\end{equation}
\end{lemma}
\begin{proof}
We would like to show that for $1 \le i \le 6$,
\begin{equation}\label{eq:boundComponent}
    P\left(U_t^{(i)}(x) \in B_{\delta_1}(0), U_s^{(i)}(y) \in B_{\delta_2}(0)\right) \lesssim \delta_1\delta_2\left(|t-s|^{1/2} + |x-y| \right)^{-1/2}.
\end{equation}
Then \eqref{eq:boundForSmallTime} follows from the independence of coordinates $U^{(i)}$. 

Consider the centered Gaussian vector $(X, Y) = \left(U_t^{(i)}(x), U_s^{(i)}(y) \right)$. Without loss of generality, suppose $t \ge s$. Since $X$ is centered Gaussian, $P\left(\mu + X \in B_{\delta_1}(0)\right)$ is maximized when $\mu = 0$. Thus, 
\begin{align}
    P\left(X \in B_{\delta_1}(0) \mid Y \right) &\lesssim \delta_1 \text{Var}\left(X - E(X \mid Y) \right)^{-1/2}. \nnend
    \intertext{Hence,}
    P\left(X \in B_{\delta_1}(0), Y \in B_{\delta_2}(0) \right) &\lesssim \delta_1\delta_2 \text{Var}\left(X - E(X \mid Y) \right)^{-1/2}. \label{eq:boundTheProb}
\end{align}
Setting $\sigma_X^2 = E(X^2), \sigma_Y^2 = E(Y^2)$, and $\rho_{X, Y}^2 = E((X-Y)^2)$, then
\begin{equation}\label{eq:ExpressCondVariance}
    \text{Var}\left(X - E(X \mid Y) \right) = \dfrac{\left(\rho_{X, Y}^2 - (\sigma_X - \sigma_Y)^2 \right) \left( (\sigma_X + \sigma_Y)^2 - \rho_{X, Y}^2 \right)}{4\sigma_X^2}.
\end{equation}
From \eqref{eq:ExpressVariance},
\begin{equation*}
     \sigma_X^2 = t^{1/2} F\left(|x|t^{-1/2}\right) \ge t^{1/2} (2\pi)^{-1/2} \ge (2\pi)^{-1/2},
\end{equation*}
and 
\begin{equation*}
    \sigma_X^2 = t^{1/2} F\left(|x|t^{-1/2}\right) \lesssim t^{1/2}|x|t^{-1/2} = |x| \le 2.
\end{equation*}
we see that $\sigma_X^2$ is bounded and bounded away from $0$. From \eqref{eq:BoundVariance}, $\rho_{X, Y}^2 \ge c_1 (|t-s|^{1/2} + |x-y|)$. Setting $\tilde{F}: \mathbb{R}^2 \rightarrow \mathbb{R}$ such that $\tilde{F}(t, x) = \left(t^{1/2}F(|x|t^{-1/2})\right)^{1/2}$, then since $F$ is differentiable, using the Mean Value Theorem, we get that $|\sigma_X - \sigma_Y| = |\tilde{F}(t, x) - \tilde{F}(s, y)| \lesssim |t-s| + |x-y|$. All these, together with \eqref{eq:ExpressCondVariance}, shows that there exists $C, \epsilon > 0$ such that 
\begin{equation}\label{eq:smallDev}
    \text{Var}\left(X - E(X \mid Y) \right) \ge C\left(|t-s|^{1/2} + |x-y|\right)
\end{equation}
for $t, s \in [1, 2]$, $|x|, |y| \le 2$, and $|t-s| + |x-y| \le \epsilon$.

Note that $\text{Var}\left(X - E(X \mid Y) \right)\left(|t-s|^{1/2} + |x-y|\right)^{-1}$ is a continuous function of $s, t, x, y$ in their specified region when $|t-s| + |x-y| \ge \epsilon$. We can then remove the constraint $|t-s| + |x-y| < \epsilon$ and modify the constant $C$ such that the bound in \eqref{eq:smallDev} holds. Combining this with \eqref{eq:boundTheProb} gives us \eqref{eq:boundComponent}, finishing our proof. 
\end{proof}
\begin{proof}[Proof of Lemma \ref{lem:boundForTwoPoints}]
Suppose the first condition in the lemma holds. Then by scaling,
\begin{align*}
    P&\left(U_t(x) \in B_{\delta t^{-\alpha}}(0), U_{t+s}(x+y) \in B_{\delta (t+s)^{-\alpha}}(0) \right) \nnend 
    = P&\left(U_1(x/t^{1/2}) \in B_{\delta t^{-\alpha-1/4}}(0), U_{1+s/t}((x+y)/t^{1/2}) \in B_{\delta (t+s)^{-\alpha}t^{-1/4}}(0) \right) \nnend 
    \intertext{By Lemma \ref{lem:boundForSmallTime}, the above is bounded by a constant multiple of}
    & \lesssim \left(\delta t^{-\alpha-1/4}\delta (t+s)^{-\alpha}t^{-1/4}\right)^6\left((s/t)^{1/2} + |y|/t^{1/2} \right)^{-3} \nnend 
    & = \delta^{12}\dfrac{t^{-6\alpha}(t+s)^{-6\alpha}}{\left(s^{1/2} + |y| \right)^3t^{3/2}} \nnend 
    \intertext{Since $|x| \le 2t^{1/2}$, the above is bounded by a constant multiple of}
    & \lesssim \delta^{12}\dfrac{t^{-6\alpha}(t+s)^{-6\alpha}}{\left(s^{1/2} + |y| \right)^3\left(t^{1/2} + |x| \right)^3},  
\end{align*}
concluding our proof when the first condition in the lemma holds.
Suppose now that the second condition in the lemma holds. From Section \ref{sec:intro}, 
\begin{equation*}
    U_{t+s}(x+y) = \int G_s(x+y-z)U_t(z)dz + \int_0^s \int G_{s-r}(x+y-z)W(dz dr).
\end{equation*}
Thus, 
\begin{align*}
    \text{Var} &\left(U_{t+s}(x+y) - E(U_{t+s}(x+y)|\mathcal{F}_t \right) \nnend 
    &= \text{Var}\left( \int_0^s \int G_{s-r}(x+y-z)W(dz dr) \right) \nnend 
    &= Cs^{1/2}.
\end{align*}
Hence, for $|y| \le 2s^{1/2}$,
\begin{equation*}
    P\left(U_{t+s}(x+y) \in B_{\delta (t+s)^{-\alpha}}(0) \mid \mathcal{F}_t \right) \lesssim \delta^6 (t+s)^{-6\alpha} s^{-3/2} \lesssim \delta^6(t+s)^{-6\alpha} \left(s^{1/2} + |y| \right)^{-3}.
\end{equation*}
Then 
\begin{equation*}
    P\left(U_t(x) \in B_{\delta t^{-\alpha}}(0), U_{t+s}(x+y) \in B_{\delta (t+s)^{-\alpha}}(0) \right) \lesssim \delta^{12}\dfrac{(t+s)^{-6\alpha} t^{-6\alpha}}{\left(s^{1/2} + |y| \right)^3\left(t^{1/2} + |x| \right)^3},
\end{equation*}
concluding our proof.
\end{proof}


\begin{thebibliography}{99}

    \bibitem{C82} D. Cox, On the existence of natural rate of escape functions for infinite dimensional Brownian motions, {\it The Annals of Probability}, {\bf 10}, 3, 623-638 (1982).
    \bibitem{DKE09} R. Dalang, D. Khoshnevisan, E. Nualart, Hitting probabilities for systems for non-linear stochastic heat equations with multiplicative noise, {\it Probability Theory and Related Fields}, {\bf 144}, 371–427 (2009).
    \bibitem{DMX17} R. Dalang, C. Mueller, Y. Xiao, Polarity of points for Gaussian random fields, {\it The Annals of Probability}, {\bf 45}, 6B, 4700-4751 (2017).
    \bibitem{DMX21} R. Dalang, C. Mueller, Y. Xiao, Polarity of almost all points for systems of non-linear stochastic heat equations in the critical dimension, {\it The Annals of Probability}, {\bf 49}, 5, 2573-2598 (2021).
    \bibitem{DE86} M. Doi, S. F. Edwards, {\it The Theory of Polymer Dynamics}, Oxford University Press, Walton Street, Oxford, 1986.
    \bibitem{E80} K. B. Erickson, Rates of escape of infinite dimensional Brownian motion, {\it The Annals of Probability}, {\bf 8}, 2, 325-338 (1980).
    \bibitem{F83} T. Funaki, Random motion of strings and related stochastic evolution equations, {\it Nagoya Mathematical Journal}, 83, 129-193 (1983).
    \bibitem{KS99} D. Khoshnevisan, Z. Shi, Brownian sheet and capacity, {\it The Annals of Probability}, {\bf 27}, 11, 1135-1159 (1999). 
    \bibitem{MT02} C. Mueller, R. Tribe, Hitting properties of a random string, {\it Electronic Journal of Probability} {\bf 7}, 1-29 (2002).
    \bibitem{MP10} P. Mörters, Y. Peres, {\it Brownian motion}, Cambridge University Press, Cambridge, 2010. 
    \bibitem{OP73} S. Orey, W. E. Pruitt, Sample functions of the $N$-parameter Wiener process, {\it The Annals of Probability}, {\bf 1}, 1, 138-163 (1973).
    \bibitem{S58} F. Spitzer, Some theorems concerning 2-dimensional Brownian motion. {\it Transactions of the American Mathematical Society} {\bf 87}, 187-197 (1958).

\end{thebibliography}
\end{document}